\documentclass{amsart}

\usepackage{amssymb,amsmath,amsthm,color}

\usepackage[bookmarksopen=true,bookmarksnumbered=true, bookmarksopenlevel=2, colorlinks=true,linkcolor=mblue,
citecolor=mblue,urlcolor=mblue, pdfstartview=FitH]{hyperref}

\definecolor{mblue}{rgb}{0,0,.8}

\usepackage[all]{xy}

\theoremstyle{plain}
\newtheorem{thm}{Theorem}[section]
\newtheorem{lemma}[thm]{Lemma}
\newtheorem{proposition}[thm]{Proposition}
\newtheorem*{prop}{Proposition}
\newtheorem{corollary}[thm]{Corollary}

\theoremstyle{definition}
\newtheorem{definition}[thm]{Definition}

\theoremstyle{remark}
\newtheorem{remark}[thm]{Remark}

\numberwithin{equation}{section}

\newcommand{\N}{{\mathfrak N}}

\newcommand{\ind}{\operatorname{Ind}}

\newcommand{\gl}{\operatorname{GL}}

\newcommand{\slll}{\operatorname{SL}}

\newcommand{\f}{\mathbb F}

\newcommand{\diag}{\operatorname{diag}}

\title[On systems of Hecke eigenvalues]{On systems of Hecke eigenvalues in cohomology of certain subgroups of $\gl_n(F)$}
\author{Morten S. Larsen}
\address{Department of Mathematics, University of Copenhagen, Universitetsparken 5, DK-2100 Copenhagen \O ,
Denmark.}
\email{\href{mailto:msl@math.ku.dk}{msl@math.ku.dk}}
\begin{document}

\begin{abstract}
We show how there is a natural action on the cohomology groups attached to certain subgroups of $\gl_n(F)$ of the Hecke operators defined as elements in an adelic double coset algebra.

Our main result is, that if a system of eigenvalues for Hecke operators occur in the cohomology groups with coefficients in certain modules, then by changing the groups the system also occur in the cohomology groups with coefficients in a $1$-dimensional module.
\end{abstract}

\maketitle


\noindent
In the following $F$ will be a number field. The ring of integers in $F$ will be denoted $\mathcal O_F$, and $C$ will denote the class group of $F$. We denote by $\f$ any given fixed field. For a ring $R$, we will denote the multiplicative group of invertible elements by $R^*$. Every module here considered will be right modules, though at appropriate places these may be considered as left modules with inverse action.

The Hecke operators will be defined as elements of an double coset algebra for a group $K$.
The group $K$ will be certain open subgroups of $\gl_n(\mathbb A)$, where $\mathbb A$ is the adeles of $F$, and we show that a double coset algebra of $K$ has a natural action on the cohomology $\bigoplus_{c\in C} H^i(\Gamma_c,V)$ for some subgroups $\Gamma_c\subseteq \gl_n(F)$ indexed by the class group $C$ and a coefficient module $V$.
This double coset algebra will have the corresponding properties to the algebras considered in \cite[\S 3.2]{shi} for $F=\mathbb Q$ the rational numbers. The need to consider the adeles arises due to nontrivial class group.

This was done in \cite{byg} for n = 2, who extended a method pioneered
and first developed by J.E. Cremona, \cite{cre}. Thus, the formalism of
Hecke operators as here defined is inspired by those works,
and can in fact also give an action on such maps on lattices, and this
in such a way that the two definitions agree.

If we only consider the part of the double coset algebra corresponding to principal ideals, then we get an action on $H^i(\Gamma_c,V)$. A system of eigenvalues is `almost' determined by this restricted action.
By only considering the action on $H^i(\Gamma_c,V)$ we will be able to use ideas from \cite{ash}. That is, we shall define Hecke pairs of level $\mathfrak N$ for an ideal $\mathfrak N$ of $F$ and admissible modules $V$.
The module $V$ will be a finite dimensional vector space over $\f$,
and our main result and primary motivation is to show that any system
of eigenvalues occurring in the cohomology with coefficients V also
occur in the cohomology with 1-dimensional coefficients.
We give two results, one valid for any field $\f$, and the other valid only when $\f$ has characteristic $\ell>0$.

The motivation to investigate how systems of eigenvalues occur in the cohomology for different coefficient modules comes from the case $F=\mathbb Q$ and $n=2$. As in \cite{as2} this corresponds to congruence properties between Hecke eigenvalues for modular forms of different level and weight.
To a system of Hecke eigenvalues for a modular form is attached a Galois representation, such that the characteristic polynomial for the Frobenius element of almost all primes is determined by the Hecke eigenvalues. In \cite{ash} it is then conjectured that any system of eigenvalues occurring in the cohomology with coefficient an admissible module correspond to a Galois representation, $F=\mathbb Q$. To what extend this is valid for any number field $F$ would be interesting to find out.

\section{Double cosets}\label{ddc}\noindent
Here we bring a summary of double cosets and the corresponding maps on the cohomology groups.

Let $G$ be a group, and let $\Gamma$ and $\Gamma'$
be subgroups of $G$. The groups $\Gamma$ and $\Gamma'$ are said to
be commensurable if $\Gamma\cap \Gamma'$ has finite index in both
$\Gamma$ and $\Gamma'$. Let $\Delta$ be a set of elements $\delta
\in G$ such that $\Gamma$ is commensurable to $\delta\Gamma'
\delta^{-1}$. We will call the triple $(\Gamma,\Gamma',\Delta)$ for a Hecke triple.
Denote with $R(\Gamma,\Gamma',\Delta)$ the space of
finite $\mathbb Z$-linear combinations of double cosets $\Gamma
\delta \Gamma'$, $\delta\in \Delta$. Equivalent to the assumption of
commensurability is that $\Gamma \delta \Gamma'= \bigcup \Gamma
\delta_j $ with a finite number of $\delta_j\in G$. Hence
$R(\Gamma,\Gamma',\Delta)$ can also be described as the space of
finite $\mathbb Z$-linear combinations of cosets $\Gamma \delta_j$,
$\delta_j\in \Delta \Gamma'$, which is right invariant under
multiplication with $\gamma'\in \Gamma'$. Hence we will note an element in $R(\Gamma,\Gamma',\Delta)$ simply as $\sum a_j\Gamma\delta_j$ a finite sum with $a_j\in \mathbb Z$. It makes no difference if we assume $\Delta=\Gamma\Delta \Gamma'$, so this will be in effect.

Let $\Gamma''$ be a third subgroup of $G$, and  $\Delta'$ be a
subset of $G$ such that $(\Gamma',\Gamma'',\Delta')$ is a Hecke triple.
There is a composition
\begin{eqnarray}\label{compdc}
R(\Gamma,\Gamma',\Delta)\times R(\Gamma',\Gamma'',\Delta') \to
R(\Gamma,\Gamma'',\Delta\Delta')
\end{eqnarray}
given by
$$ (\sum a_j \Gamma \delta_j,\sum a'_i\Gamma' \delta'_i) \mapsto \sum a_ja'_i \Gamma  \delta_j\delta'_i
$$
If $\Gamma'=\Gamma$ then we write
$R(\Gamma,\Delta)=R(\Gamma,\Gamma,\Delta)$.
If $\Delta$ is a Semigroup we call $(\Gamma,\Delta)$ for a Hecke pair, and $R(\Gamma,\Delta)$ is an $\mathbb Z$-algebra.

Let $V$ be a module for some semigroup containing $\Delta$, $\Gamma$ and $\Gamma'$, then the elements of $R(\Gamma,\Gamma',\Delta)$ correspond naturally to maps on the cohomology groups
$$H^i(\Gamma,V)\to H^i(\Gamma',V)$$
given in the following way.
Let $\tilde f\in H^i(\Gamma,V)$ and $f:\Gamma^{i+1}\to V$ be a homogeneous $i$-cocycle lying in the cohomology class $\tilde f$, then  $\sum a_j\Gamma \delta_j \in R(\Gamma,\Gamma', \Delta)$ determines a homogeneous $i$-cocycle $f': \Gamma'^{i+1}\to V$ given by
$$f'(\gamma_0',\ldots, \gamma_i')=\sum_j a_j f(t_j(\gamma_0'),\ldots,t_j(\gamma_i'))\delta_j$$
Here $\gamma_0',\ldots,\gamma_i'\in\Gamma'$ and $t_j:\Gamma'\to \Gamma$ is determined such that $t_j(\gamma')\delta_r=\delta_j\gamma'$ for some $r$. The cohomology class of $f'$ is uniquely determined by $\sum \Gamma\delta_j$ and will be denoted $\tilde f|\sum a_j \Gamma \delta_j$. For details in the case $\Gamma=\Gamma'$ see \cite{rw} or \cite{kps}.

If we also have the Hecke triple $(\Gamma',\Gamma'',\Delta')$, and assuming that $V$ is a module for some semigroup also containing  $\Gamma''$ and $\Delta'$, then composition of maps on the cohomology groups coincides with the composition of the double cosets \eqref{compdc}. In particular if $(\Gamma,\Delta)$ is a Hecke pair then $H^i(\Gamma,V)$ is a $R(\Gamma,\Delta)$-module.

For a Hecke pair $(\Gamma,\Delta)$ the long exact cohomology sequence is a exact sequence of $R(\Gamma,\Delta)$-modules. From \cite{as1} we need the notion of compatibility.
\begin{definition} Two Hecke pairs $(\Gamma,\Delta)$ and $(\Gamma',\Delta')$ are compatible if $\Gamma\subseteq \Gamma'$, $\Delta\subseteq \Delta '$ and
\begin{itemize}
\item $\Gamma'\cap \Delta \Delta^{-1}= \Gamma$,
\item $\Gamma' \Delta = \Delta'$.
\end{itemize}
\end{definition}

If $V$ is a $\Delta$-module and $[\Gamma':\Gamma]<\infty$ then we will consider the induced $\Gamma'$-module $\ind(\Gamma,\Gamma',V)$ as the set of functions $f:\Gamma' \to V$ satisfying $f(\gamma\gamma')=f(\gamma')\gamma^{-1}$ for every $\gamma\in \Gamma$, $\gamma'\in \Gamma'$. The action of $\tilde \gamma'\in \Gamma'$ is given by $(f\tilde \gamma')(\gamma')=f(\gamma'\tilde \gamma'^{-1})$.

If $(\Gamma,\Delta)\hookrightarrow(\Gamma' ,\Delta')$ are compatible Hecke pairs, then there is an injective homomorphism $R(\Gamma',\Delta')\to R(\Gamma,\Delta)$. The induced module $\ind(\Gamma,\Gamma',V)$ is naturally a $\Delta'$-module. The action of $\delta'\in\Delta'$ given by
 $$(f\delta')(\gamma')=f(\hat \gamma ')\delta$$
with $\delta\in\Delta$ and $\hat\gamma'\in \Gamma'$ satisfying $\gamma'\delta'^{-1}=\delta^{-1}\hat \gamma'$.
This makes the Shapiro isomorphism an isomorphism of $R(\Gamma',\Delta')$-modules.

\begin{lemma}\label{conj} Let $g,g'\in G$ and $V$ a right $G$ module. Then there is a bijection
$R(\Gamma,\Gamma',\Delta)\to R(g\Gamma g^{-1},g'\Gamma' g'^{-1}, g\Delta g'^{-1})$,
$T\mapsto \hat T$, and isomorphisms $H^i(\Gamma,V)\simeq H^i(g\Gamma g^{-1},V)$,
$H^i(\Gamma',V)\simeq H^i(g'\Gamma' g'^{-1}, V)$ such that the following diagram
commutes
$$ \begin{matrix} H^i(\Gamma,V) & \simeq & H^i(g\Gamma g^{-1}, V) \\
\downarrow T & &\downarrow \hat T \\
H^i(\Gamma', V) & \simeq & H^i(g'\Gamma'g'^{-1}, V)
\end{matrix}$$
\end{lemma}

\begin{proof}
The mapping $R(\Gamma,\Gamma', \Delta) \to R(g\Gamma g^{-1},g'\Gamma'
g'^{-1}, g\Delta g'^{-1})$ is given by $\Gamma \delta \Gamma' \mapsto
 g \Gamma \delta \Gamma' g'^{-1}$.

Let $f: \Gamma^{i+1}\to V$ be an homogeneous $i$-cocycle. Define the
$i$-cocycle $\hat f: (g\Gamma g^{-1})^{i+1}\to V$ by
$\hat f( \overline \gamma)=f(g^{-1}\overline \gamma g) g^{-1}$. This
gives an isomorphism commuting with the coboundary, hence an
isomorphism of the cohomology groups. In the same way we have an isomorphism $H^i(\Gamma',V)\simeq H^i(g'\Gamma' g'^{-1},V)$ where an $i$-cocycle $f'$ correspond to $\hat f'$ given by $\hat f'(\overline{\gamma'})=f'(g'^{-1}\overline \gamma' g') g'^{-1}$.

Let $T=\sum a_j \Gamma\delta_j \in R(\Gamma,\Gamma',\Delta)$, then $\hat
T=\sum a_j g\Gamma g^{-1} g\delta_jg'^{-1}$. Writing down the maps, it is immediately seen that the given diagram commutes.
\end{proof}

\section{Hecke action on cohomology groups}\label{heckeaction}\noindent
We will now for certain adelic Hecke pairs $(K,\Delta)$ describe how to find subgroups $\Gamma_c \subseteq \gl_n(F)$, $c\in C$, such that there is a natural action of the algebra $R(K,\Delta)$ on $\bigoplus_{c\in C} H^i(\Gamma_c, V)$ for some module $V$.

Let $\mathbb A$ denote the adeles of $F$. For $\alpha\in \gl_1(\mathbb A)$ we will  write $(\alpha)$ for the fractional ideal determined by $\alpha$ which only depends upon the finite part.
Denote by $F_\infty$ the infinite part of $\mathbb A$. Let $F_\infty^*=\gl_1(F_\infty)$ and $F_\infty^+$ be the subgroup of $F_\infty^*$ consisting of the totally positive elements. Meaning for every real prime the elements of $F_\infty^+$ are positive. For $\alpha\in \gl_1(\mathbb A)$ the image in $F_\infty^*/F_\infty^+$ of $\alpha$ will be called the sign of $\alpha$. If $\alpha$ has the same sign as $1$ then $\alpha$ is totally positive. For any set $S\subseteq \gl_n(\mathbb A)$ let $S^+$ denote the subset of $S$ consisting of elements with totally positive determinant.
The only role the infinity part of $\gl_n(\mathbb A)$ will play here, is so we can keep track of the sign of the determinant.

In the following $(K,\Delta)$ will be a Hecke pair $K,\Delta\subseteq \gl_n(\mathbb A)$. Furthermore we assume that $(\kappa)$ is a principal ideal for every $\kappa\in K$.

Set $\Delta^{(c)}=\{\delta \in \Delta| (\det \delta)\in c\}$ for $c\in C$.
Then $R(K,\Delta)=\bigcup_{c\in C} R(K,\Delta^{(c)})$ as a graded algebra, with the subalgebra $R(K,\Delta^{(1)})$.

A system of eigenvalues for $R(K,\Delta)$ is an algebra homomorphism $\Phi:R(K,\Delta)\to \f$.
Let $J$ be the group of ideals for $F$, and $\chi:J\to \f^*$ a homomorphism. Define
$\chi\otimes \Phi$ by $\chi \otimes \Phi(K\delta K) = \chi((\det \delta)) \Phi(K\delta K)$. Then $\chi\otimes \Phi$ is also an algebra homomorphism.

\begin{lemma}\label{unramifiedtwist} Let $\Phi,\Psi:R(K,\Delta)\to \f$ be algebra homomorphisms. If $\Phi$ and
$\Psi$ is equal on $R(K,\Delta_1)$, then there exist an character $\chi:C\to \f^*$
such that $\Phi=\chi\otimes \Psi$.
\end{lemma}
\begin{proof}
Let $K\delta K\in R(K,\Delta)$, and note that $\Phi(K\delta K)=0$ if and only if
$\Psi(K\delta K)=0$. Hence we can find a map $\chi': \Delta \to \f^*$ such that
$\Phi(K\delta K)=\chi'(\delta)\Psi(K\delta K)$.

Let $\delta_1,\delta_2 \in \Delta$ be such that $(\det \delta_1)$ is in the same ideal class
as $(\det \delta_2)$. Choose $\delta\in \Delta$ such that $(\det \delta \delta_1)$ is principal and such that $\Phi (K \delta K)\neq 0$ if $\Phi(K\delta_j K)\neq 0$ for $j=1$ or $j=2$.
Then $$\Psi(K\delta K\cdot K\delta_j K)
=\Phi(K\delta K) \Phi(K\delta_j K)= \chi'(\delta)\chi'(\delta_j) \Psi(K\delta K \cdot
K\delta_j K )$$
and hence we may choose $\chi'$ to satisfy $\chi'(\delta_j)=\chi'(\delta)^{-1}$ for $j=1,2$. So $\chi'$ depends only on the ideal class of the determinant.

Let now $\delta_1$ and $\delta_2$ be any elements of $\Delta$, and write $K\delta_1 K\cdot K\delta_2 K= \sum K\delta_i K$.
Then $(\det \delta_i)=(\det \delta_1)(\det \delta_2)$ for every $\alpha$, so
\begin{eqnarray}\nonumber
\chi'(\delta_1)\chi'(\delta_2)\Psi(K\delta_1K)\Psi(K\delta_2 K)=\Phi(K\delta_1 K \cdot K\delta_2 K)=\sum \Phi(K \delta_i K) \\
=\chi'(\delta_1\delta_2)(\sum \Psi(K\delta_i K))=\chi'(\delta_1\delta_2)\Psi(K\delta_1K )\Psi(K\delta_2 K)
\nonumber\end{eqnarray}
Hence we may choose $\chi'$ such that $\chi'(\delta_1)\chi'(\delta_2)=\chi'(\delta_1\delta_2)$.
The map $\chi: C\to \f^*$ determined by $\chi'$ is then a homomorphism and satisfies $\Phi=\chi\otimes \Psi$.
\end{proof}

Fix elements $\alpha_c\in \gl_n(\mathbb A)$ for each $c\in C$ such that $(\det \alpha_c)\in c$.
We shall consider open subgroups $K$ of $\gl_n(\mathbb A)$ such that the determinant is surjective
$$\bigcup_{c\in C} \alpha_c \gl_n(F)K \to \gl_1(\mathbb A)$$
The reason is the following proposition.

\begin{proposition} \label{strongap}
Let $K$ be an open subgroup of $\gl_n(\mathbb A)$.
Then we have the disjoint union
\begin{eqnarray}\label{fulddet}\nonumber
\bigcup_{c\in C}\alpha_c \gl_n(F) K = \gl_n(\mathbb A)
\end{eqnarray}
if and only if the determinant
$\cup_{c\in C} \alpha_c \gl_n(F)K \to \gl_1(\mathbb A)$
 is surjective.

\end{proposition}
\begin{proof}
The `only if' is clear. We show the `if' part.
From strong approximation for $\operatorname{SL}_n$ we have, \cite{kn},
\begin{eqnarray} \label{stsl}
\operatorname{SL}_n(F)(\operatorname{SL}_n(\mathbb A)\cap K)=\operatorname{SL} _n(\mathbb A)
\end{eqnarray}

Let $\beta \in \gl_n(\mathbb A)$. By assumption we have a $c\in C$, $\alpha\in \gl_n(F)$, and $\kappa\in K$ such that $\det \beta= \det (\alpha_c \alpha \kappa)$. So $\alpha^{-1} \alpha_c^{-1} \beta \kappa^{-1}\in \operatorname{SL}_n(\mathbb A)$. By \eqref{stsl} we have $\tilde \alpha \in \gl_n(F)$ and $\tilde \kappa\in K$ so
$\alpha^{-1} \alpha_c^{-1} \beta \kappa^{-1}=\tilde \alpha \tilde \kappa$
giving
$$\beta= \alpha_c \alpha \tilde \alpha \tilde \kappa \kappa \in \alpha_c \gl_n(F) K$$
The union is disjoint by the standing assumption on $K$, that $(\kappa)$ is principal for every $\kappa \in K$
\end{proof}

Notice that with the assumption of proposition \ref{strongap} if $\beta_1,\beta_2 \in \gl_n(\mathbb A)$ with $\beta_1\beta_2\in \alpha_c \gl_n(F) K$ then
$$\alpha_c\gl_n(F)K =\beta_1\gl_n(F)\beta_2 K = K\beta_1 \gl_n(F)\beta_2$$

For every $c\in C$ set $\widetilde \Gamma_{c}= K\cap \alpha_c \gl_n(F) \alpha_c^{-1}$ and $\Gamma_c=\alpha_c^{-1} \widetilde \Gamma_{c} \alpha_c =\alpha_c^{-1} K\alpha_c \cap \gl_n(F)$. For every $c,c'\in C$ let $\widetilde \Delta_{c,c'}= \Delta\cap \alpha_c \gl_n(F) \alpha_{c'}^{-1}$ and $\Delta_{c,c'}=\alpha_c^{-1}\widetilde \Delta_{c,c'} \alpha_{c'}=\alpha_c^{-1} \Delta \alpha_{c'} \cap \gl_n(F)$.
\begin{proposition}  \label{nedigen}
Let $(K,\Delta)$ be a Hecke pair with the group $K$ such that the determinant $\cup_{c\in C} \alpha _c \gl_n(F)K\to \gl_1(\mathbb A)$ is surjective.
With the notation as above, there is an injection
\begin{eqnarray}\label{dennemor}
R(K,\Delta^{(c)})\to R(\widetilde \Gamma_{c'}, \widetilde \Gamma_{c'c}, \widetilde \Delta_{c',c'c})\simeq R(\Gamma_{c'}, \Gamma_{c'c}, \Delta_{c',c'c})
\end{eqnarray}
for every $c,c'\in C$ given by
$$\sum a_j K \delta_j \mapsto \sum a_j K\delta_j \cap \widetilde\Delta_{c',c'c} \textrm{ resp. } \sum a_j K\delta_j \mapsto \sum a_j \alpha_{c'}^{-1}K \delta_j\alpha_{c'c} \cap \Delta_{c',c'c}$$
These maps  form together an injective algebra homomorphism
\begin{eqnarray} \nonumber R(K,\Delta)=\bigoplus_{c\in C}R(K,\Delta^{(c)}) &\to & \bigoplus_{c\in C}\prod_{c'\in C} R(\widetilde \Gamma_{c'},\widetilde \Gamma_{c'c},\widetilde \Delta_{c',c'c})
\\&\simeq& \bigoplus_{c\in C}\prod_{c'\in C} R(\Gamma_{c'}, \Gamma_{c'c}, \Delta_{c',c'c})
\nonumber
\end{eqnarray}
\end{proposition}
\begin{proof} Recall that $\Delta^{(c)}=\{\delta \in \Delta| (\det \delta)\in c\}$. By proposition \ref{strongap} the intersection $K\delta \cap\widetilde \Delta_{c',c'c}$ is not empty exactly when $\delta\in \Delta^{(c)}$.
Hence for $K\delta K \in R(K,\Delta^{(c)})$ we may assume  $\delta\in \widetilde \Delta_{c',c'c}$. Let $\sum a_j K\delta_j\in R(K,\Delta^{(c)})$ with $\delta_j\in \widetilde \Delta_{c',c'c}$.
For every $\widetilde \gamma \in \widetilde \Gamma_{c'c}$ and every $j$ we have $\delta_j \widetilde \gamma= \kappa_j\gamma_i$ for some $\kappa_j\in K$ and some $i$.
For $\kappa\in K$ and $\delta\in \widetilde \Delta_{c',c'c}$ we have that $\kappa \delta\in \widetilde \Delta_{c',c'c}$ precisely when $\kappa\in \widetilde \Gamma_{c'}$, hence $\kappa_j\in \widetilde \Gamma_{c'}$. This shows that $\sum a_j\widetilde \Gamma_{c'}\delta_j$ is an element of $R(\widetilde \Gamma_{c'},\widetilde \Gamma_{c'c},\Delta_{c',c'c})$. So we do have a mapping, and it is clearly injective. Note that if $c$ is the principal ideal class, then this is the same as saying that $(\Gamma_{c'},\Delta_{c',c'})\hookrightarrow (K,\Delta^{(1)})$ are compatible Hecke pairs.

To check these maps form an algebra homomorphism, we need to check that it respects products. This is equivalent to showing that this diagram commutes
$$\xymatrix{ R(K, \Delta^{(c)})\oplus R(K, \Delta^{(c')}) \ar[r] \ar[d] & R(\widetilde \Gamma_{c''},\widetilde\Gamma_{c''c}, \widetilde \Delta_{c'',c''c})\oplus R(\widetilde \Gamma_{c''c},\widetilde \Gamma_{c''cc'},\widetilde \Delta_{c''c,c''cc'}) \ar[d]\\
R(K,  \Delta^{(cc')}) \ar[r] & R(\widetilde \Gamma_{c''},\widetilde \Gamma_{c''cc'}, \widetilde \Delta_{c'',c''cc'})
}$$
for every $c,c',c'' \in C$. Here the horizontal maps are those described above, and the vertical maps are the product. For $\sum a_j K\delta_j \in R(K,\Delta^{(c)})$ and
$\sum a_i K\delta_i\in R(K,\Delta^{(c')})$
we may assume $\delta_j\in \widetilde \Delta_{c'',c''c}$ and $\delta_i\in \widetilde \Delta_{c''c,c''cc'}$.
Going down and right gives $(\sum a_j K \delta_j, \sum a_i K\delta_i)\mapsto \sum a_ja_i K \delta_j\delta_i \mapsto \sum a_j a_i \widetilde \Gamma_{c''}\delta_j\delta_i$,
and going right and then down gives $(\sum a_j K\delta_j,\sum a_i K\delta_i ) \mapsto (\sum a_j \widetilde \Gamma_{c''} \delta_j, \sum a_i \widetilde \Gamma_{c''c} \delta_i)\mapsto \sum a_ja_i \widetilde \Gamma_{c''} \delta_j \delta_i$.

From lemma \ref{conj} we have $R(\widetilde \Gamma_{c'}, \widetilde \Gamma_{c'c}, \widetilde \Delta_{c',c'c})\simeq R(\Gamma_{c'}, \Gamma_{c'c}, \Delta_{c',c'c})$.
\end{proof}

The notation from proposition \ref{nedigen} will be used from now on.
If $V$ is a right module for some semigroup containing $\Delta_{c,c'}$ for every $c,c'\in C$ then we can consider $\bigoplus_{c\in C}H^r(\Gamma_{c},V)$
with the action of $\bigoplus_{c''\in C}\prod_{c'\in C} R(\Gamma_{c'}, \Gamma_{c'c''}, \Delta_{c',c'c''})$.
Using the homomorphism from proposition \ref{nedigen} we now get an action of $R(K,\Delta)$ on
$\bigoplus_{c\in C}H^r(\Gamma_{c},V)$. This action will simply be denoted $f|(\sum a_jK\delta_j)$, where $f\in \bigoplus_{c\in C}H^r(\Gamma_{c},V)$, $\sum a_j K\delta_j \in R(K,\Delta)$.

\begin{remark} Let $(K,\Delta)$ be a Hecke pair such that the determinant is surjective $\cup_{c\in C} \alpha_c \gl_n(F) K\to \gl_1(\mathbb A)$.
The action here defined does in the following sense not depend on the initial choice of fixed elements $\alpha_c$, $c\in C$.
If $\beta_c$ was another choice then by proposition \ref{strongap} we have $\beta_c = \kappa_c \alpha_c g_c$ for some $\kappa_c\in K$ and $g_c\in \gl_n(F)$. So
$$\gl_n(F)\cap \beta_c^{-1} \Delta \beta_{c'}= g_c^{-1}(\gl_n(F)\cap \alpha_c^{-1}
\Delta \alpha_{c'})g_{c'}=g_c^{-1} \Delta_{c,c'} g_{c'}$$
and
$$ \gl_n(F) \cap \beta_c^{-1} K\beta_{c'}=g_c^{-1} \Gamma_c g_c^{-1}$$
Hence if $V$ is a right module for some group containing $g_c$ and $\Delta_{c,c'}$ for every $c,c'\in C$,
then we see from lemma \ref{conj} that using the $\beta_c$ instead of $\alpha_c$ doesn't change the action of $R(K,\Delta)$ on
$\bigoplus_{c\in C} H^r (\Gamma_c,V)\simeq \bigoplus_{c\in C} H^r (g_c^{-1}\Gamma_cg_c,V)$.
\end{remark}

\begin{definition} Let $\Phi:R(K,\Delta)\to \f^*$ be a system of eigenvalues. Then $\Phi$ is said to occur in $\bigoplus_{c\in C} H^i(\Gamma_c,V)$, if there is a eigenform $f\in \bigoplus_{c\in C} H^i(\Gamma_c,V)$ such that
$$f|T=\Phi(T)f, \textrm{ for all } T\in R(K,\Delta)$$
\end{definition}

\begin{proposition} Let $(K,\Delta)$ be a Hecke pair with $\cup_{c\in C}\alpha_c \gl_n(F)K\to \gl_1(\mathbb A)$ surjective. Suppose $V$ is a module for some semigroup containing $\Delta_{c,c'}$ for every $c,c'\in C$,
and $\chi:C\to \f^*$ is a homomorphism. If $\Phi: R(K,\Delta)\to \f$ is a system of eigenvalues occurring in
$\bigoplus_{c\in C} H^i(\Gamma_c, V)$, then so does $\chi\otimes \Phi$.
\end{proposition}
\begin{proof}
Let $f=\sum f_c\in \bigoplus H^i(\Gamma_c, V)$, $f_c\in H^i(\Gamma_c,V)$, be the eigenform corresponding to $\Phi$. For $\delta \in \Delta$ with $(\det \delta) \in c'$ the component of
$f|K\delta K$ in $H^i(\Gamma_c,V)$ is
$$(f|K\delta K)_c = f_{cc'^{-1}}|K\delta K=\Phi(K\delta K) f_c$$
Set $f^\chi= \sum \chi(c)^{-1} f_c$. Then the component of $f^\chi|K\delta K$ in $H^i(\Gamma_c,V)$ is
\begin{eqnarray} \nonumber
(f^\chi|K\delta K)_c &=& (f^\chi)_{cc'^{-1}}|K\delta K= \chi(c')\chi(c)^{-1} f_{cc'^{-1}} |K \delta K \\
\nonumber
&=&\Phi(K\delta K) \chi(c')\chi(c)^{-1} f_c= \Phi(K\delta K) \chi( c') (f^\chi)_{c}
\end{eqnarray}
Hence $f^\chi$ is an eigenform corresponding to the system of eigenvalues $\chi \otimes \Phi$.
\end{proof}
Combining this with lemma \ref{unramifiedtwist} gives:
\begin{corollary}\label{cortwist}
If $\Psi:R(K,\Delta)\to \f$ is an algebra homomorphism, and the restriction of $\Psi$ to 
$R(K,\Delta^{(1)})$
is equal to a system of eigenvalues of $R(K,\Delta^{(1)})$
 acting on $H^i(\Gamma_{c'},V)$ for any $c'\in C$,
then $\Psi$ occurs
in $\bigoplus_{c\in C} H^i(\Gamma_c,V)$.
\end{corollary}

\section{The Hecke operators}
\noindent
In this section we define the Hecke operators. The algebra of the Hecke operators is an adelic double coset algebras with properties generalizing the double coset algebras considered in \cite[\S 3.2]{shi} for $\mathbb Q$.
The Hecke operators defined in \cite{weil} can be seen to be the action of such a double coset algebra.

For a prime ideal $\mathfrak p$ of $F$ let $\mathcal O_{\mathfrak p}$ be the $\mathfrak p$-adic integers and $\pi_{\mathfrak p}$ an uniformizer.
Let $\mathbb A_f$ denote the finite adeles of $F$. For an element $\alpha\in \gl_1(\mathbb A_f)$ let $(\alpha)$ denote the fractional ideal determined by $\alpha$.
 For any adelic object a subscript of $\mathfrak p$ will denote the $\mathfrak p$-part.
Let $K_0=\prod \gl_n(\mathcal O_{\mathfrak p})$ a subgroup in $\gl_n(\mathbb A_f)$, and let $\Delta_0$ be the semigroup consisting of $\alpha\in \gl_n(\mathbb A_f)$ with $\alpha_{\mathfrak p}$ a matrix with entries in $\mathcal O_{\mathfrak p}$ for every prime ideal $\mathfrak p$.

Consider the algebra $R(K_0,\Delta_0)$.
For a prime ideal $\mathfrak p$ and an integer $1\leq m \leq n$ set
$$T_{\mathfrak p}^{(m)}=K_0\diag(1,\ldots,1,
\pi_{\mathfrak p},\ldots ,\pi_{\mathfrak p} )K_0\qquad \textrm{(with
$m$ $\pi_{\mathfrak p}$ on the diagonal)}
$$
and for any ideal $\mathfrak a$ of $F$ let $T_{\mathfrak a}$ be the element in $R(K_0,\Delta_0)$ given by the
sum of the different double cosets of the form $K_0 \alpha K_0$ with $( \det
\alpha )=\mathfrak a$.

The elements $T_{\mathfrak p}^{(m)}$ will be called the Hecke operators.
The double coset algebras considered in section \ref{udenniveau} will be homomorphic images of $R(K_0,\Delta_0)$. We do not consider Hecke operators for primes dividing the level, so the kernel will be generated by $T_{\mathfrak p}^{(m)}$ for all $m$ and some $\mathfrak p$.

Notice that $R(K_0,\Delta_0)$ is the restricted tensorproduct of the algebras $R(K_{0,\mathfrak p},\Delta_{0,\mathfrak p})$. From this and \cite[Remark 3.25]{shi} the following proposition follows, though we will not need this fact.
\begin{prop} The ring $R(K_0,\Delta_0)$ is a
polynomial ring over $\mathbb Z$ in the indeterminates $T_{\mathfrak
p}^{(m)}$, $1\leq m \leq n$ and every prime ideal $\mathfrak p$ of $F$.
We have
$$\deg (T_{\mathfrak p}^{(m)})=\frac{(N(\mathfrak
p)^n-1)(N(\mathfrak p) ^{n} -N(\mathfrak p))\cdots (N(\mathfrak
p)^n- N(\mathfrak p)^{m-1})}{ (N(\mathfrak p)^m-1)(N(\mathfrak p)^m-
N(\mathfrak p))\cdots (N(\mathfrak p)^m -N(\mathfrak p)^{m-1})}$$
For each prime ideal $\mathfrak p$ we have the formal identity
$$\sum_{k=0}^\infty T_{\mathfrak p^k} X^k=\left(\sum_{j=0}^n(-1)^j
N(\mathfrak p)^{j(j-1)/2}T_{\mathfrak p}^{(j)} X^j\right)^{-1}  $$
\end{prop}

\section{Reduction to $1$-dimensional coefficients}\label{udenniveau}
\noindent
The result of corollary \ref{cortwist} says, that if we for a certain Hecke pair $(K,\Delta)$ want to consider the systems of eigenvalues for $R(K,\Delta)$ occurring in $\bigoplus_{c\in C} H^i(\Gamma_c,V)$, then it is sufficient to consider the systems of eigenvalues for $R(\Gamma_c,\Delta_{c,c})$ occurring in $H^i(\Gamma_c,V)$ for any $c\in C$. This will allow us to transfer the methods of \cite{ash} to this more general context.

Fix an ideal $\mathfrak N$ of $F$. Let $\pi$ denote the projection to $\gl_n(\mathcal O_F/\mathfrak N)$ from the subgroup of $\gl_n(\mathbb A)$ consisting of the elements such that the $\mathfrak p$'th component is in $\gl_n(\mathcal O_{\mathfrak p})$ for every prime $\mathfrak p|\mathfrak N$. Let $\widehat \pi$ denote the projection to $F_\infty^*/F_\infty^+\times \gl_n(\mathcal O_F/\mathfrak N)$ where the image in $F_\infty^*/F_\infty^+$ is the sign of the determinant, and the image in $\gl_n(\mathcal O_F/\mathfrak N)$ is determined by $\pi$.

Fix furthermore an element $\alpha_c\in \gl_n(\mathbb A)$ for every $c$ in the class group $C$ of $F$ such that $(\det \alpha_c)\in c$ and such that the infinity component and the $\mathfrak p$'th component of $\alpha _c$ is trivial, for every prime $\mathfrak p|\mathfrak N$.

Set $K_1=\gl_n(F_\infty)\times K_0$ and $\Delta_1=\gl_n(F_\infty)\times \Delta_0$, and consider the following

\begin{eqnarray}\nonumber
\Delta_{\mathfrak N}&=& \{\delta \in \Delta_1| (\det \delta) \textrm{ is prime
to
} \mathfrak N\},
\quad \phantom{nn} \Delta_{\mathfrak N, c,c'} = \gl_n(F)\cap
\alpha_{c}^{-1} \Delta_{\mathfrak N} \alpha_{c'} \\
\nonumber \Delta(\mathfrak N)&=& \{\delta \in \Delta_{\mathfrak N}|
\pi(\delta)=\diag(1,\ldots, 1, *)\}, \quad \Delta(\mathfrak N)_{c,c'}=\gl_n(F)\cap
\alpha_c^{-1} \Delta(\mathfrak N) \alpha_{c'} \\
\nonumber K(\mathfrak N)&=& K_1\cap \Delta(\mathfrak N),\qquad
\phantom{nnnnnnnnnnnnnnn} \Gamma(\mathfrak N)_c=
\gl_n(F)\cap \alpha_c^{-1} K(\mathfrak N) \alpha_{c'}
\end{eqnarray}

\begin{lemma} \label{lemnivn}
The Hecke pair $(K(\mathfrak N), \Delta(\mathfrak N))$ is
compatible with $(K_1,\Delta_{\mathfrak N})$.
\end{lemma}
\begin{proof}
The elements in $\Delta_{\mathfrak N}$ is elements of $\Delta_1$ with the restriction that
the $\mathfrak p$'th component lies in $\gl_n(\mathcal O_{\mathfrak p})$ for every $\mathfrak p|\mathfrak N$.
So for $\delta \in \Delta_{\mathfrak N}$ the  $\mathfrak p$'th component $\delta_{\mathfrak p}$ considered as an element of $\gl_n(\mathbb A)$ lies in $K_1$.
Likewise considerations can be made for $\Delta(\mathfrak N)$ and $K(\mathfrak N)$.
Giving $K_1 \Delta(\mathfrak N)=\Delta_{\mathfrak N}$ and $K_1\cap \Delta(\mathfrak N)\Delta(\mathfrak N)^{-1} = K(\mathfrak N)$.
\end{proof}
\begin{definition} A Hecke pair $(K,\Delta)$ is of level $\mathfrak N$ if
$(K(\mathfrak N),\Delta(\mathfrak N))\hookrightarrow (K,\Delta)$ and
$(K,\Delta)\hookrightarrow (K_{1},\Delta_{\mathfrak N} )$ are
compatible.
\end{definition}
For a Hecke pair $(K,\Delta)$ of level $\mathfrak N$ we have  isomorphisms  $R(K,\Delta)\simeq R(K_1,\Delta_{\mathfrak N})\simeq R(K(\mathfrak n),\Delta(\mathfrak N))$. Furthermore $R(K,\Delta)$ is the image of $R(K_1,\Delta_1)\simeq R(K_0,\Delta_0)$ with kernel generated by $T_{\mathfrak p}^{(m)}$ for $\mathfrak p|\mathfrak N$.

If $(K,\Delta)$ is a Hecke pair of level $\mathfrak N$, then $(\kappa)$ is a principal ideal for every $\kappa\in K$, and the determinant $\cup_{c\in C} \alpha_c \gl_n(F)K\to \gl_1(\mathbb A)$ is surjective. Hence the results of section \ref{heckeaction} is applicable.

\begin{remark}\label{undergrp}
Note that we have $K(\mathfrak N)
= \pi^{-1}( \pi (K(\mathfrak N))) \cap K_1$.
If $(K,\Delta)$ is a Hecke pair of level $\mathfrak N$ then $\pi(K)$ is a subgroup of $\gl_n(\mathcal O_F/\mathfrak N)$ containing $\pi(K(\mathfrak N))$.
Furthermore we have $K=\pi^{-1}(\pi(K))\cap K_1$ since if $\kappa_1\in K_1$ and $\kappa\in K$ with $\pi(\kappa_1)=\pi(\kappa)$ then by the above $\kappa_1\kappa^{-1}\in K(\mathfrak N)\subseteq K$.
By considerations as in the proof of lemma \ref{lemnivn} we have $\Delta= \pi^{-1}(\pi(K))\cap \Delta_{\mathfrak N}$.

On the other hand, if $H\subseteq \gl_n(\mathcal O_F/\mathfrak N)$ is a subgroup containing $\pi(K(\mathfrak N))$ then setting $K=\pi^{-1}(H)\cap K_1$ and $\Delta= \pi^{-1}(H)\cap \Delta_{\mathfrak N}$ gives a Hecke pair $(K,\Delta)$ of level $\mathfrak N$.
Hence Hecke pairs of level $\mathfrak N$ correspond to such groups $H$.

Passing to $\Gamma_c$ and $\Delta_{c,c'}$ we see that these too are determined as subgroup of $\Gamma(1)_{c}$ resp. subset of $\Delta_{\mathfrak N,c,c'}$ by the image of $\pi$. In particular, if an element $\gamma\in \Gamma(1)_{c}$ is congruent to an element of $\Delta_{c,c}$ mod $\mathfrak N$, then $\gamma\in \Gamma_c$.
\end{remark}

\begin{remark}
Denote by $P_{\mathfrak N}$ the group of principal fractional ideals of $F$ prime to $\mathfrak N$, and by $P_{\mathfrak N}(1)$ the subgroup consisting of those principal ideals generated by a totally positive element in $F$ congruent to 1 mod $\mathfrak N$.
Consider the group $P_{\mathfrak N}/P_{\mathfrak N}(1)$. Let an element $a\in \mathcal O_F^*$ prime to $\mathfrak N$ map to the ideal $(a)$ in $P_{\mathfrak N} /P_{\mathfrak N}(1)$. Note that this map factors through $F_\infty^*/F_\infty^+\times (\mathcal O_F/\mathfrak N)^*$, hence there is a natural surjection $F_\infty^*/F_\infty^+\times (\mathcal O_F/\mathfrak N)^*\to P_{\mathfrak N}/P_{\mathfrak N}(1)$. We will consider the determinant as $K\stackrel{\det}{\to} F_\infty^*\times \prod_{\mathfrak p} \mathcal O_{\mathfrak p}^*\to F_\infty^*/F_\infty^+\times (\mathcal O_F/\mathfrak N)^*$.

We may extend the definition of Hecke pairs of level $N$ in the following way. Let $U$ be a subgroup of $F_\infty^*/F_\infty^+\times (\mathcal O_F/\mathfrak N)^*$ such that the restriction $U\to P_{\mathfrak N}/P_{\mathfrak N}(1)$ is surjective. Then  a Hecke pair $(K,\Delta)$ is of level $\mathfrak N$ and determinant group $U$ if there is a Hecke pair $(K',\Delta')$ of level $\mathfrak N$ such that $K$ consist of the elements of $K'$ with determinant mapped to $U$, and $\Delta$ is the elements of $\Delta'$ likewise with determinant in $U$. Any result here proven for Hecke pairs of level $\mathfrak N$ can easily be seen valid for Hecke pairs of level $\mathfrak N$ with determinant group $U$ for any group $U$ as above. For $\mathbb Q$ this could for example be the subgroup of positive elements.

Note also that $\widehat \pi(\Delta_{c,c})/\widehat \pi(\Gamma_c)\simeq P_{\mathfrak N}/P_{\mathfrak N}(1)$ for a Hecke pair $(K,\Delta)$ of level $\mathfrak N$. This is so since an element of $\widehat \pi(\Delta_{c,c})$ lies in $\widehat \pi(\Gamma_c)$ if and only if the 'determinant' does. This follows from lemma \ref{compatgamma} below and the obvious fact for $\Delta(\mathfrak N)_{c,c}$ and $\Gamma(\mathfrak N)_{c}$.
\end{remark}
\begin{remark} For a Hecke pair $(K,\Delta)$ of level $\mathfrak N$, it is possible to show that the maps \eqref{dennemor} also are surjections. This can be achieved by showing that for every $\delta\in \widetilde \Delta_{c',c'c}$ we have $K\delta K=\cup K\delta \widetilde \gamma_j$ for some $\widetilde \gamma_j\in \widetilde\Gamma_{c'c}$. Whether two cosets $K\delta \kappa_1$ and $K\delta\kappa_2$ are equal, only depends upon the image of $\kappa_1$ and $\kappa_2$ in $\gl_n(\mathcal O_F/ (\det \delta))$. In particular it is sufficient to consider elements that maps to $\slll_n(\mathcal O_F/(\det \delta))$, which follows from the analogue to the $\mathbb Q$ case in \cite{shi}. The conclusion now follows, since the map $\widetilde\Gamma_{c'c}\cap \slll_n(\mathbb A)\to \slll_n(\mathcal O_F/(\det \delta))$ is surjective.
\end{remark}

\begin{lemma}\label{compatgamma}
If $(K,\Delta)\hookrightarrow (K',\Delta ')$ are compatible Hecke pairs of level $\mathfrak N$,
then the Hecke pairs $(\Gamma_c,\Delta_{c,c})\hookrightarrow (\Gamma_c',\Delta_{c,c}')$ are compatible.
\end{lemma}
\begin{proof}
The first condition is satisfied, $\Gamma_c=\gl_n(F)\cap \alpha^{-1}_c K \alpha_c\subseteq \gl_n(F)\cap\alpha^{-1}_c K' \alpha_c =\Gamma_c'$
and $\Delta_{c,c}=\gl_n(F)\cap \alpha_c^{-1}\Delta\alpha_c \subseteq \gl_n(F)\cap
\alpha_c^{-1} \Delta ' \alpha_c=\Delta_{c,c}'$. For the second condition we have $\Gamma_c \subseteq \Gamma_c' \cap \Delta_{c,c}\Delta_{c,c}^{-1}$.
The other inclusion follows since an element of $\Gamma_c'$ congruent mod $\mathfrak N$ to an element in $\Delta_{c,c}$ lies in $\Gamma_c$ by remark \ref{undergrp}.
For the third we have $\Gamma'_c\Delta_{c,c}\subseteq \Delta_{c,c}'$. Since $\Delta_{c,c}'\subseteq \alpha_c^{-1}\Delta'\alpha_c = \alpha_c^{-1} K' \Delta\alpha_c$
we may write $\delta'\in \Delta_{c,c}'$ as $\delta'= \kappa' \delta $ with $\kappa' \in \alpha_c^{-1} K' \alpha_c$ and $\delta \in \alpha_c^{-1} \Delta \alpha_c$.
By theorem \ref{strongap} we have $\delta\in \alpha_{c}^{-1}K'\alpha_c\kappa '\delta \gl_n(F)
=\alpha_c^{-1}K\alpha_c \gl_n(F)$.
So there is $\kappa \in \alpha_c^{-1} K\alpha_c$ with
$\kappa \delta \in \gl_n(F)\cap \alpha_c^{-1}\Delta\alpha_c=\Delta_{c,c}$. Hence $\delta' = (\kappa '\kappa^{-1}) (\kappa \delta) \in \Gamma_c' \Delta_{c,c}$
giving the other inclusion.
\end{proof}

Recall that $\alpha_c$ is chosen such that the $\mathfrak p$'th component is trivial for every $\mathfrak p|\mathfrak N$, and so $\pi(K)=\pi(\Delta_{c,c'}) \subseteq
\gl_n(\mathcal O_F/\mathfrak N)$ for $c,c'\in C$. For every $c,c'\in C$ we have a homomorphism $\Delta_{c,c'}\to F_\infty^*/F_\infty^+\times \pi(K)$ determined by $\widehat \pi$. That is image in $F_\infty^*/F_\infty^+$ is the sign of the infinity component of the determinant, and the image in $\pi(K)$ is given by $\pi$.
\begin{definition}\label{addef}
Let $V$ be a finite dimensional vector space over $\f$, and $(K,\Delta)$ a Hecke pair. If $V$ is a right $F_\infty^*/F_\infty^+\times \pi(K)$ module, we call $V$ an admissible module for $(K,\Delta)$.

\end{definition}
Hence when $V$ is an admissible module for $(K,\Delta)$ we have a natural action of $\Delta_{c,c'}$ on $V$ for every $c,c'\in C$. An action where $\Delta_{c,c'}^+$ acts through its image in $\gl_n(\mathcal O_F/\mathfrak N)$.

Let $\chi: F_\infty^*/F_\infty^+\times (\mathcal O_F/\mathfrak N)^*\to \f^*$ be a character. Denote also by $\chi$ the composite with $F_\infty ^*/F_\infty^+\times \gl_n(\mathcal O_F/\mathfrak N)\to F_\infty^*/F_\infty^+\times (\mathcal O_F /\mathfrak N)^*$. If $V$ is an admissible module for a Hecke pair $(K,\Delta)$ of level $\mathfrak N$, let $V(\chi)$ denote the admissible module such that we have a vector space isomorphism $\iota :V\to V(\chi)$ and
$$\iota(v)\hat \kappa = \chi(\hat \kappa) \iota(v\hat \kappa), \quad \hat \kappa \in F_\infty^*/F_\infty^+\times \pi(K)$$
\begin{lemma}\label{pos}
Let $(K,\Delta)$ be a Hecke pair of level $\mathfrak N$.
If $V$ is a finite $\f$ vector space with a right action of $\Delta_{c,c'}$ for any $c,c'\in C$ such that the action of $\Delta_{c,c'}^+$ only depends upon the image in $\gl_n(\mathcal O_F/\mathfrak N)$, then $V$ is in a natural way an admissible module.
\end{lemma}
\begin{proof}
The homomorphism $\Delta_{c,c'}\to F_\infty^*/F_\infty^+\times \pi(K)$ is surjective.
This is seen since it is possible to find an element $a\in \mathcal O_F$ of any given sign and $a\equiv 1 \pmod{\mathfrak N}$. So $\diag(1,\ldots ,1, a) \in \Delta(\mathfrak N)\subseteq \Delta$. Hence the action of $\Delta_{c,c'}^+$ on $V$ determines the action of $\pi(K)$, and the action of $F_\infty^*/F_\infty^+$ is given by $v\delta \pi(\delta)^{-1}$, $v\in V$, which only depends upon the sign of $\delta$.
\end{proof}

\begin{lemma}
Let $(K,\Delta)\hookrightarrow(K',\Delta')$ be compatible Hecke
pairs
and let $c\in C$ be given. If $V$ is an
admissible module for $(K,\Delta)$, then
$I=\ind(\Gamma_c,\Gamma'_c,V)$ is an admissible
module for $(K',\Delta')$.
\end{lemma}
\begin{proof}
Using lemma \ref{pos} we need only show that $\Delta '^+_{c,c}$ acts through its image in $\gl_n(\mathcal O_F/\mathfrak N)$, where the action is given as in section \ref{ddc}.

For $f\in I$ and $\delta_1',\delta_2'\in \Delta'^+_{c,c}$ we will show
that $f\delta_1'=f\delta_2'$ if $\pi(\delta_1')=\pi( \delta_2')$. For any $\gamma' \in \Gamma'_c$
find $\delta_j\in \Delta_{c,c}$ and
$\gamma_j'\in \Gamma_c'$ such that
$\delta_j'\gamma'^{-1}=\gamma'^{-1}_j \delta_j$, $j=1,2$. So
$\pi(\gamma'^{-1}_1\delta_1)=\pi(\gamma'^{-1}_2\delta_2) $ and by remark \ref{undergrp} we get
$\gamma'_1\gamma'^{-1}_2\in
\Gamma_c$. Since $\gamma'_1\gamma'^{-1}_2$
has total positive
determinant, we get
\begin{eqnarray}\nonumber
f\delta_1'(\gamma')&=&f(\gamma'_1)\delta_1=f(\gamma'_1)\pi(\gamma'_1\gamma'^{-1}_2\delta_2\delta^{-1}_1)\delta_1
=f(\gamma'_1)\gamma'_1\gamma'^{-1}_2\delta_2\pi(\delta_1^{-1}\delta_1)
\\&=&f(\gamma'_2\gamma_1'^{-1}\gamma'_1)\delta_2=f\delta_2'(\gamma')
\nonumber\end{eqnarray}
hence $f\delta_1'=f\delta_2'$.
\end{proof}

\begin{lemma}\label{inductionlemma}
Let $(K,\Delta)\hookrightarrow (K',\Delta')$ be compatible Hecke pairs, fix $c\in C$, and let $V$ be an admissible module for $(K,\Delta)$. The module $I=\ind(\Gamma_c,\Gamma'_c,V)$ is as an admissible module isomorphic to $J=\ind(\widehat \pi(K),\widehat \pi(K'),V)$.
\end{lemma}
\begin{proof} Let $g\in J$. For $\kappa'\in K'$ there is $\delta\in \Delta_{c,c}$ and $\gamma'\in \Gamma'_c$ such that $\widehat \pi(\kappa')=\widehat \pi(\delta)^{-1} \widehat \pi(\gamma')$.
Since $g(\widehat \pi(\delta)^{-1}\widehat \pi(\gamma') )=g(\widehat \pi(\gamma'))\widehat \pi(\delta)$ the element $g\in J$ is uniquely determined by its restriction (as a map) to $\widehat \pi(\Gamma'_c)$.

On the other hand $f\in I$ as a map of $\Gamma'_c$ into $V$ then $f(\gamma')$ only depends upon $\widehat \pi(\gamma')$, $\gamma'\in \Gamma'_c$. This is seen to give a correspondence between $I$ and $J$. Writing up the given action, it is immediately seen to be the same.
\end{proof}

\begin{thm} \label{thmreduc}
Let $(K,\Delta)$ be a Hecke pair of level $\N$, and $V$ an admissible module for $(K,\Delta)$. If
$\Phi:R(K,\Delta)\to \f$ is a system of eigenvalues occurring in $\bigoplus_{c\in C} H^i(\Gamma_c,V)$, then
there is a character
$\chi: F_\infty^*/F_\infty^+\times (\mathcal O_F/\mathfrak N)^*\to \f^*$
 such
that $\Phi$
occurs in $\bigoplus_{c\in C} H^j(\Gamma(\N)_c,
\f(\chi))$ for some $j\leq i$.
\end{thm}
\begin{proof} Note first that any irreducible $\widehat \pi(K(\N))$ is of the form $\f(\chi)$ for some $\chi: F_\infty^*/F_\infty^+\times (\mathcal O_F/\mathfrak N)^*\to \f^*$.

Using corollary \ref{cortwist} we only need to consider the action of $R(K,\Delta^{(1)})=R(\Gamma_c,\Delta_{c,c})$ on $H^i(\Gamma_c,V)$ for any fixed $c\in C$. By lemma \ref{compatgamma} the Hecke pairs $(\Gamma(\N)_c,\Delta(\N)_{c,c})\hookrightarrow (\Gamma_c, \Delta_{c,c})$ are compatible, and hence the action respects short exact sequences, and we may assume $V$ irreducible.

So $V$ is an irreducible module for the finite group $\widehat \pi (K)$. Consider the restriction of $V$ to $\widehat \pi (K(\N))$ and let $W$ be an irreducible quotient of this $\widehat \pi(K(\N))$-module. Then $V$ is a submodule of $\ind(\widehat \pi(K(\N)),\widehat \pi(K),W)\simeq \ind(\Gamma(\N)_c,\Gamma_c,W)$ (using lemma \ref{inductionlemma}), e.g. we have an exact sequence
$$0\to V \to \ind (\Gamma(\N)_c,\Gamma_c,W) \to X \to 0$$
The long exact sequence of cohomology is compatible with the action of $R(K,\Delta^{(1)})$. Hence we have an exact sequence of $R(K,\Delta^{(1)})$-modules
$$H^{i-1}(\Gamma_c, X)\to H^i(\Gamma_c, V)\to H^i(\Gamma_c, \ind (\Gamma(\N)_c,\Gamma_c, W))$$
The Shapiro isomorphism $H^i(\Gamma_c, \ind (\Gamma(\N)_c,\Gamma_c, W))\simeq H^i(\Gamma(\N)_c,W)$ is compatible with the action of the Hecke operators, hence if the given eigenvector in $H^i(\Gamma_c, V)$ maps nonzero into $H^i(\Gamma_c, \ind (\Gamma(\N)_c,\Gamma_c, W))$ we are done. If this is not the case, then  \cite[proposition 1.2.2]{as1} says that the system of eigenvalues occurs in $H^{i-1}(\Gamma_c, X)$. The result follows by induction on the cohomology dimension.
\end{proof}

\section{Characteristic $\ell$}
Let $\ell$ denote a prime number. In the following $\f$ will denote a field of characteristic $\ell$. Irreducible representations on $\f$-vector spaces might all be $1$-dimensional for a nonabelian group. This gives results which can be seen as generalizing results from \cite{as2}.

First we note some representation results.

\begin{lemma} \label{semi}
Let $G$ be a finite group and $V$ a semisimple $\f G$-module. If $H$ is a normal subgroup of $G$, then the restriction of $V$ to $H$ is still a semisimple module. \end{lemma}
\begin{proof}
Assume $V$ an irreducible $G$ module.
Let $W$ be an irreducible component of $V$ considered as a $H$-module. Since $H$ is normal in $G$, $Wg$ is likewise an irreducible $H$-submodule of $V$ for every $g\in G$.
The $G$-module $V$ is irreducible, so $V =\sum_{g\in G} Wg$ showing that $V$ as a $H$-module must be semisimple.
\end{proof}
\begin{lemma}\label{through}
Let $\lambda$ be a prime of $F$, $\lambda |\ell$. An irreducible representation of $\gl_n(\mathcal O_F/\lambda^\nu)$ on a $\f$-vector space factors through $\gl_n(\mathcal O_F/\lambda)$, $\nu$ any natural number.
\end{lemma}
\begin{proof} The kernel of the map $\gl_n(\mathcal O_F/\lambda^\nu)\to \gl_n(\mathcal O_F/\lambda)$ is an $\ell$-group.
By \ref{semi} the restriction of an irreducible representation to this kernel is still semisimple, and hence the kernel acts trivial.
\end{proof}

\begin{lemma}\label{subH}
Let $\mathbb F_\lambda $ be a finite field of characteristic $\ell$.
Let $H$ be the subgroup of $\gl_n(\mathbb F_\lambda)$ consisting of upper triangular matrices with the first $n-1$ diagonal entries equal to $1$.
Any irreducible $\f H$-module is $1$-dimensional and the action factors through the determinant.
\end{lemma}
\begin{proof}
The subgroup of $H$ consisting of upper triangular matrices with all diagonal entries equal to $1$ is a normal subgroup and an $\ell$-group.
By lemma \ref{semi} this subgroup acts trivial on a irreducible $\f H$-module.
\end{proof}

Set $\mathfrak L= \prod_{\lambda |\ell} \lambda$ where $\lambda$ is a prime of $F$. In the following $\mathfrak N$ is an ideal of $F$ prime to $\ell$, and $\nu$ will be some natural number.

\begin{thm} \label{thmred}
Let $(K,\Delta)$  be a Hecke pair of level $\mathfrak N \mathfrak L^\nu$ such that the image of $K$ in $\gl_n(\mathbb F_\lambda)$ is contained in the subgroup $H$ of lemma \ref{subH} for every prime $\lambda$ of $F$, $\lambda |\ell$ and $\mathbb F_\lambda =\mathcal O_F/\lambda$.
Suppose $(K',\Delta')$ is a Hecke pair of level $\mathfrak N\mathfrak L^\nu$ such that the image of $K$ in $\gl_n(\mathcal O_F/\mathfrak N)$ is a normal subgroup in the image of $K'$.
If $V$ is an admissible module for $(K',\Delta')$ such that $K'^+$ acts through the image in $\gl_n(\mathcal O_F/\mathfrak L^\nu)$, then a system of eigenvalues $\Phi$ occurring in $\bigoplus_{c\in C} H^i(\Gamma_c',V)$ occurs in $\bigoplus_{c\in C} H^j(\Gamma_c,\f(\chi))$ for some $j\leq i$ and some character $\chi: F_\infty^*/F_\infty^+\times (\mathcal O_F/\mathfrak L)^*\to \f^*$.
\end{thm}
\begin{proof}
Set $K''=K'\cap K$ and $\Delta'' = \Delta' \cap \Delta$, then $(K'',\Delta'')$ is also a Hecke pair of level $\mathfrak N \mathfrak L^\nu$.
As in the proof of theorem \ref{thmreduc} we may assume $V$ irreducible, and we can find an irreducible admissible module $W$ for $(K'',\Delta'')$ such that we have an exact sequence
$$0\to V\to \ind(\widehat \pi (K''),\widehat \pi(K'),W) \to X \to 0$$
of admissible $(K',\Delta')$ modules.
Consider this as a sequence of admissible $(K'',\Delta'')$ modules. By assumption $K''^+$ acts on $V$ through the image in $\gl_n(\mathcal O_F /\mathfrak L^\nu)$ and hence likewise on $W$, a subquotient of $V$.
Since the image of $K''$ in $\gl_n(\mathcal O_F /\mathfrak N)$ is a normal subgroup of the image of $K'$, we see that the action of $K''^+$ on $\ind(\widehat \pi(K''),\widehat\pi(K'),W)$ also only depends upon the image in $\gl_n(\mathcal O_F/\mathfrak L^\nu)$. So this is also true for $X$.

Arguing as in the proof of theorem \ref{thmreduc} we have that either the system of eigenvalues $\Phi$ occurs in $\bigoplus H^i(\Gamma_c'',W)$ or $\bigoplus H^{i-1}(\Gamma_c',X)$.
Using induction on the cohomology dimension $\Phi$ occurs in $\bigoplus H^j(\Gamma_c'', Y)$ for some admissible $(K'',\Delta'')$ module $Y$, on which $K''^+$ acts through its image in $\gl_n(\mathcal O_F/\mathfrak L^\nu)$, and some $j\leq i$.

The Shapiro isomorphism gives that $\Phi$ occurs in $\bigoplus H^j(\Gamma_c,\ind(\widehat \pi(K''), \widehat \pi (K) ,Y))$.
But then $\Phi$ likewise occurs in $\bigoplus H^j(\Gamma_c, Z)$, where $Z$ in an irreducible $\widehat \pi(K)$ module and a component in the semisimplification of $\ind(\widehat \pi(K''),\widehat \pi(K) ,Y)$.
This implies that $K^+$ acts on $Z$ through its image in $\gl_n(\mathcal O_F/\mathfrak L^\nu)$.

To finish the proof we need only notice how the module $Z$ looks like. From lemma \ref{through} we have that $K^+$ acts on $Z$ through the image in $\gl_n(\mathcal O_F/\mathfrak L)$.
By the assumption of the theorem and lemma \ref{subH} we now conclude that $Z=\f(\chi)$ for some homomorphism $\chi: F_\infty^*/F_\infty^+\times (\mathcal O_F/\mathfrak L)^*$.
\end{proof}

A more specific formulation of theorem \ref{thmred} can be given in the following way for $n=2$.
For any ideal $\mathfrak a$ of $F$ set
$$K_1(\mathfrak a)=\{\kappa\in K_1 | \kappa \equiv\left(\begin{smallmatrix} 1 & *\\ 0& * \end{smallmatrix} \right)\pmod{\mathfrak a}\}, \quad \Gamma_1(\mathfrak a)_c=\gl_n(F)\cap \alpha_c^{-1} K_1(\mathfrak a) \alpha _c$$
Let $\widetilde {\mathfrak L}$ be
an ideal of $F$ dividing $\mathfrak L^\nu$ for some $\nu\geq 1$.
Then we have a Hecke pair $(K_1(\mathfrak N \widetilde L),K_1(\mathfrak N\widetilde {\mathfrak L})\Delta(\mathfrak N\mathfrak L^\nu))$ of level $\mathfrak N\mathfrak L^\nu$.
\begin{corollary}
Let $V$ be a finite dimensional $\f$ vector space and a $F_\infty /F_\infty ^+ \times \gl_n(\mathcal O_F/\mathfrak L)$ module. If the system of eigenvalues $\Phi$ occurs in $\bigoplus_{c\in C} H^i(\Gamma_1(\mathfrak N\widetilde {\mathfrak L})_c,V)$, then for some $j\leq i$ and character $\chi: F_\infty^* /F_\infty^+\times (\mathcal O_F/\mathfrak L)^*\to \f^*$ the system of eigenvalues $\Phi$ occurs in $\bigoplus_{c\in C} H^j(\Gamma_1(\mathfrak N \mathfrak L)_c,\f(\chi))$.
\end{corollary}


\begin{thebibliography}{Hupprts}

\bibitem[Ash92]{ash}
A. Ash: \emph{Galois representations attached to mod $p$ cohomology of ${\rm GL}(n,Z)$},
Duke Math. J. 65 (1992), no. 2, 235--255

\bibitem[AS86a]{as1}
A. Ash, G. Stevens: \emph{Cohomology of arithmetic groups and congruences between systems of Hecke eigenvalues}, J. Reine Angew. Math. 365 (1986), 192--220

\bibitem[AS86b]{as2}
A. Ash, G. Stevens: \emph{Modular forms in characteristic $l$ and special values of their $L$-functions}, Duke Math. J. 53 (1986), no. 3, 849--868

\bibitem[Byg]{byg}
J.S. Bygott, \emph{Modular forms and elliptic curves over imaginary quadratic number fields}, PhD thesis, University of Exeter, 1998

\bibitem[Cre]{cre}
J.E. Cremona, \emph{Hyperbolic tessellations, modular symbols, and elliptic curves over complex quadratic fields}, Compositia Mathematica, 51 (1984), 275--323

\bibitem[Kn66]{kn}
M. Kneser: \emph{Strong approximation}, in:
A. Borel, G. Mostow (eds.): Algebraic groups and discontinuous
subgroups,
Proc. Sympos. Pure Math. 9 (1966), 187--196.

\bibitem[KPS81]{kps}
M. Kuga, W. Parry, C.-H. Sah, \emph{Group cohomology and Hecke operators}, in: Manifolds and Lie Groups, Progress in mathematics 14, Birkh\"{a}user, 1981

\bibitem[RW70]{rw}
Y.H. Rhie, G. Whaples, \emph{Hecke operators in cohomology of groups}, J. Math. Soc. Japan 22 (1970), 431--442

\bibitem[Shi]{shi} G. Shimura: \emph{Introduction to the Arithmetic Theory of Automorphic
Functions}, Princeton University Press, 1971

\bibitem[Weil]{weil} A. Weil: \emph{Dirichlet Series and Automorphic Forms}, Lecture Notes in Mathematics 189, Springer-Verlag, 1971

\end{thebibliography}
\end{document}